\documentclass[12pt,reqno]{amsart}

\usepackage{etoolbox}
\makeatletter
\patchcmd{\@sect}{\@tocwrite{#1}{#8}}{\@tocwrite{#1}{#7}}{}{}
\makeatother

\usepackage[T1]{fontenc}
\usepackage{lmodern}
\usepackage[margin=1in]{geometry}
\usepackage{amsmath,amssymb,mathtools}
\usepackage{microtype}
\usepackage[colorlinks=true,linkcolor=blue,pagebackref]{hyperref}
\usepackage{enumerate}
\usepackage{needspace}

\newtheorem{theorem}{Theorem}
\newtheorem{lemma}[theorem]{Lemma}
\newtheorem{proposition}[theorem]{Proposition}
\newtheorem{corollary}[theorem]{Corollary}
\theoremstyle{definition}

\newtheorem{remark}[theorem]{Remark}

\DeclareMathOperator{\Id}{Id}
\DeclareMathOperator{\Spec}{Spec}
\DeclareMathOperator{\Prim}{Prim}
\newcommand{\rcf}{\mathrm{rcf}}
\DeclareMathOperator{\diag}{diag}
\newcommand{\reg}{\operatorname{reg}}
\newcommand{\im}{\operatorname{im}}

\newcommand{\fin}{\mathrm{fin}}
\newcommand{\N}{\mathbb N_0}

\title{A note on the ideal lattice of a weighted Leavitt path algebra}
\author{Raimund Preusser}
\address{School of Mathematics and Statistics, Nanjing University of Information Science and Technology, Nanjing, China}  
\email{raimund.preusser@gmx.de}
\date{}

\begin{document}
\begin{abstract}
Let $K$ be a field and $L_K(E,w)$ the Leavitt path algebra of a row-finite weighted graph $(E,w)$. Suppose that $L_K(E,w)$ is not isomorphic to an ordinary Leavitt path algebra. We prove that the ideal lattice of every countably generated unital $K$-algebra embeds as a closed interval into the ideal lattice of $L_K(E,w)$. The embedding preserves multiplication of ideals modulo the lower endpoint of the interval. It follows that the ideal lattice of $L_K(E,w)$ is not distributive and that ideal multiplication in $L_K(E,w)$ is not commutative. Moreover, we show that if $(E,w)$ is countable, then $L_K(E,w)$ has exactly $|K|^{\aleph_0}$ ideals, prime ideals, left primitive ideals, and right primitive ideals.
\end{abstract}

\maketitle
\tableofcontents

\section{Introduction}
The ideal structure of ordinary Leavitt path algebras (of directed graphs) is well understood \cite{AAM,Rangaswamy}. In particular, their two-sided ideal lattices are distributive and their ideal multiplication is commutative \cite{Rangaswamy}. For weighted Leavitt path algebras, the ideal structure is much less understood. The problem of describing their ideals was highlighted in the open problems section of the survey \cite{Survey}. 

Several contributions address aspects of this problem. Ara \cite[Section~5]{Ara} studies ideals through connections with Leavitt path algebras of separated graphs, describing idempotent-generated ideals in certain quotients of weighted Leavitt path algebras. Hazrat and Nam \cite{HazratNam} describe the lattice of vertex-generated ideals for weighted Leavitt path algebras associated with sandpile graphs equipped with their balanced weighting. More recently, Damalerio, Hazrat and Nam \cite[Theorem~3.9]{DamalerioHazratNam} show that for Leavitt path algebras of vertex-weighted graphs, an order-preserving $\mathbb Z[x,x^{-1}]$-module isomorphism of graded Grothendieck groups induces an isomorphism of the join-semilattices of vertex-generated ideals.

Our main result shows that if a weighted Leavitt path algebra over a field is not isomorphic to an ordinary Leavitt path algebra, then its ideal lattice accommodates the ideal lattice of every countably generated unital algebra over the same field. More precisely, let $L=L_K(E,w)$ be a weighted Leavitt path algebra, where $K$ is a field and $(E,w)$ is row-finite. If $L$ is not isomorphic to an ordinary Leavitt path algebra, then for every countably generated unital $K$-algebra $B$ there are ideals $Y\subseteq Z$ of $L$ and a lattice isomorphism
\begin{equation}\label{eq:main}
 \Id(B)\cong[Y,Z]\subseteq\Id(L).
\end{equation}
The isomorphism preserves ideal multiplication modulo $Y$. It follows that, in contrast to the unweighted case, the ideal lattice of $L_K(E,w)$ is not distributive and that ideal multiplication in $L_K(E,w)$ is not commutative. We also obtain an exact count of ideals and of prime and primitive ideals when the weighted graph $(E,w)$ is countable. 

We prove \eqref{eq:main} by constructing a representation $L\rightarrow\mathbb M^{\rcf}_{\Lambda}(B)$ such that the image of $L$ contains $\mathbb M^{\fin}_{\Lambda}(B)$. Here $\Lambda$ is a nonempty set, $\mathbb M^{\rcf}_{\Lambda}(B)$ is the set of all row- and column-finite matrices over $B$ whose rows and columns are indexed by $\Lambda$, and $\mathbb M^{\fin}_{\Lambda}(B)$ is the set of all matrices in $\mathbb M^{\rcf}_{\Lambda}(B)$ with finitely many nonzero entries. The existence of such a representation implies \eqref{eq:main}, see Lemma \ref{prop:matrix-transfer}.

In Section 2, we recall some basic definitions regarding lattices, weighted graphs and weighted Leavitt path algebras. Section 3 contains some results on row- and column-finite matrix algebras. In Section 4, we show how to use branching systems to obtain representations $L\rightarrow\mathbb M^{\rcf}_{X}(A)$ where $A$ is any unital $K$-algebra. In Section 5, we construct, under the assumption that $L$ is not isomorphic to an ordinary Leavitt path algebra and that the algebra $A$ is generated by two elements $a,b$ such that $b^4=1$, a representation $\rho:L\rightarrow\mathbb M^{\rcf}_{X}(A)$ such that $\mathbb M^{\fin}_{\Omega}(A)\lhd \rho(L)$ where $\Omega\subseteq X$ is some nonempty subset. In Section 6, we prove our main result Theorem \ref{thm:main}. In Section 7, we obtain the cardinalities of various sets of ideals in $L$, under the assumption that $L$ is not isomorphic to an ordinary Leavitt path algebra and that the weighted graph $(E,w)$ is countable. 

\section{Preliminaries}
Throughout the paper $K$ denotes a fixed field. All algebras are associative (but not necessarily unital) $K$-algebras. Modules over algebras with local units are understood to be unital. All ideals are two-sided, and products of ideals consist of finite sums of products.
\subsection{Lattices}
A \emph{lattice} is a partially ordered set in which every pair $x,y$ has a join (i.e. a least upper bound) $x\vee y$ and a meet (i.e. a greatest lower bound) $x\wedge y$. It is \emph{distributive} if
\[
 x\wedge(y\vee z)=(x\wedge y)\vee(x\wedge z)
\]
for all $x,y,z$ (or equivalently, if the dual distributive identity holds). A map $f$ between two lattices is called a \emph{lattice isomorphism} if it is bijective and $x\leq y\Leftrightarrow f(x)\leq f(y)$. Note that any lattice isomorphism preserves joins and meets.

For $a\leq b$, the \emph{interval} $[a,b]$ consists of all $x$ with $a\leq x\leq b$. An interval in a lattice is itself a lattice with the induced order. Write $\Id(A)$ for the lattice of two-sided ideals of an algebra $A$, ordered by inclusion, with joins given by sums and meets by intersections. In particular, an interval $[Y,Z]\subseteq \Id(A)$ consists of the ideals $I$ satisfying $Y\subseteq I\subseteq Z$. 

\subsection{Weighted graphs and Condition (LPA)}
A \emph{weighted graph} is a pair $(E,w)$, where $E=(E^0,E^1,s,r)$ is a directed graph with vertex set $E^0$, edge set $E^1$, and source and range maps $s,r:E^1\to E^0$, and $w:E^1\to\{1,2,\ldots\}$ is a weight map. A weighted graph is called \textit{row-finite} if no vertex emits infinitely many edges. 

We call an edge $e$ \emph{heavy} if $w(e)>1$. For a vertex $v$, we write $T(v)$ for the vertices reachable from $v$ by a (directed) path (here vertices are considered as paths of length $0$, so $v\in T(v)$). Two edges $e,f$ are \emph{in line} if $e=f$, or $s(f)\in T(r(e))$, or $s(e)\in T(r(f))$. A \textit{cycle} based at a vertex $v$ is a closed path $e^{(1)}\dots e^{(n)}$ of length $n\geq 1$ such that $s(e^{(1)})=v$ and $s(e^{(i)})\neq s(e^{(j)})$ if $i\neq j$. 

A weighted graph satisfies \textit{Condition (LPA)} if the conditions (LPA1)-(LPA4) below hold (cf. \cite[Introduction]{Preusser}).
\begin{enumerate}
\medskip
\item[(LPA1)] No vertex emits more than one heavy edge.
\medskip
\item[(LPA2)] No vertex in $\bigcup_{e\text{ heavy}}T(r(e))$ emits more than one edge.
\medskip
\item[(LPA3)] If heavy edges $e,f$ are not in line, then $T(r(e))\cap T(r(f))=\emptyset$.
\medskip
\item[(LPA4)] If $e$ is heavy and a cycle is based at a vertex of $T(r(e))$, then that cycle contains $e$.
\medskip
\end{enumerate}

\subsection{Weighted Leavitt path algebras}
Let $(E,w)$ be a row-finite weighted graph, and let $E^{0}_{\reg}$ denote the set of all vertices that emit at least one edge. Set $w(v)=\max\{w(e)\mid e\in s^{-1}(v)\}$ for every $v\in E^0_{\reg}$. The $K$-algebra $L_K(E,w)$ presented by the generating set 
\[\{v,e_i,e_i^*\mid v\in E^0, e\in E^1, 1\leq i\leq w(e)\}\]
and the relations
\begin{enumerate}[(W1)]
\medskip
\item $uv=\delta_{uv}u\quad(u,v\in E^0)$,
\medskip
\item $s(e)e_i=e_i=e_ir(e),~r(e)e_i^*=e_i^*=e_i^*s(e)\quad(e\in E^1, 1\leq i\leq w(e))$,
\medskip
\item 
$\sum\limits_{i=1}^{w(v)}e_i^*f_i= \delta_{ef}r(e)\quad(v\in E_{\reg}^0;~e,f\in s^{-1}(v))$,
\medskip 
\item $\sum\limits_{e\in s^{-1}(v)}e_ie_j^*= \delta_{ij}v\quad(v\in E_{\reg}^0;~1\leq i, j\leq w(v))$
\medskip
\end{enumerate}
is called the {\it weighted Leavitt path algebra} of $(E,w)$. In relations (W3) and (W4) we set $e_i$ and $e_i^*$ equal to zero whenever $i > w(e)$. Finite sums of distinct vertices form a set of local units for $L_K(E,w)$. If all edge weights are $1$, then $L_K(E,w)$ is isomorphic to the ordinary Leavitt path algebra $L_K(E)$.

By \cite[Theorems 1 and 2]{Preusser}, a weighted Leavitt path algebra $L_K(E,w)$ of a row-finite weighted graph $(E,w)$ is ring isomorphic to an ordinary Leavitt path algebra with coefficients in a field if and only if $(E,w)$ satisfies Condition (LPA).

\section{Some results on row- and column-finite matrix algebras}
For a set $X$ and a unital algebra $A$, let $\mathbb M^{\rcf}_{X}(A)$ denote the algebra of all matrices $(a_{xy})_{x,y\in X}$ over $A$ such that every row and every column has only finitely many nonzero entries. Addition and multiplication are the usual matrix operations, the identity is the matrix $I_X=(\delta_{xy}1_A)_{x,y\in X}$. Let $\mathbb M^{\fin}_{X}(A)$ denote the two-sided ideal of $\mathbb M^{\rcf}_{X}(A)$ consisting of all matrices with only finitely many nonzero entries in total. Write $E_{xy}(a)$ for the matrix with entry $a$ at $(x,y)$ and zeros elsewhere, and abbreviate $E_{xy}(1)$ to $E_{xy}$. 

A \emph{partial bijection} of $X$ is a bijection $f:Q\to R$ between subsets $Q,R\subseteq X$. Its \emph{partial permutation matrix} $P_f\in \mathbb M^{\rcf}_{X}(A)$ has entry $1_A$ at position $(x,f(x))$ for each $x\in Q$, and zeros elsewhere. Clearly $P_{f^{-1}}=P_f^T$. Moreover, if $f:Q\to R$ and $g:U\to V$ are partial bijections of $X$, and $f\circ g$ denotes their composition with domain $\{u\in U\mid g(u)\in Q\}$, then $P_{f\circ g}=P_gP_f$.

For $\Omega\subseteq X$, matrices indexed by $\Omega$ are viewed as matrices indexed by $X$ by extension by zero outside $\Omega\times\Omega$.

\begin{lemma}\label{lem:uniform-coefficients}
Let $B$ be a nonzero countably generated unital $K$-algebra. Then there are
$a,b\in\mathbb M^{\rcf}_{\N}(B)$ with $b^4=1$ such that the subalgebra of $\mathbb M^{\rcf}_{\N}(B)$ generated by $a$ and $b$ contains $ \mathbb M^{\fin}_{\N}(B)$.
\end{lemma}
\begin{proof}
Choose a sequence of generators $b_0,b_1,\ldots$ of $B$.
Let $U,V\in\mathbb M^{\rcf}_{\N}(B)$ be the matrices with entries $U_{i+1,i}=V_{i,i+1}=1$ for all $i\geq 0$ and all other entries zero. Then $VU=1$ and $UV=1-E_{00}$. Set $\Delta=\diag(b_0,b_1,\ldots)\in\mathbb M^{\rcf}_{\N}(B)$ and let $R$ be the subalgebra of $\mathbb M^{\rcf}_{\N}(B)$ generated by $U$, $V$ and $\Delta$.
Clearly
\[
 U^iE_{00}V^j=E_{ij},\qquad E_{00}V^i\Delta U^iE_{00}=b_iE_{00}
\]
and hence $R$ contains $\mathbb M^{\fin}_{\N}(B)$.

In $\mathbb M_4(R)$, write $\varepsilon_{ij}$ for the four-by-four matrix units,
and set
\begin{align*}
 a=\varepsilon_{11}+U\varepsilon_{23}
   +V\varepsilon_{24}+\Delta\varepsilon_{34}\in \mathbb M_4(R),\qquad
  b=\varepsilon_{21}+\varepsilon_{32}+\varepsilon_{43}+\varepsilon_{14}\in \mathbb M_4(R).
\end{align*}
Then $a^3=\varepsilon_{11}$ and $b^4=1$. Moreover, $\varepsilon_{ij}=b^{i-1}a^3b^{5-j}$ for all $1\leq i,j\leq 4$. Multiplying $a$ on both sides by the matrix units recovers $U,V,\Delta$
in every matrix position. Thus $a$ and $b$ generate the algebra $\mathbb M_4(R)$.

Let 
\[\alpha:\mathbb M_4(\mathbb M^{\rcf}_{\N}(B))\to \mathbb M^{\rcf}_{\{1,2,3,4\}\times\N}(B)\]
be the algebra isomorphism defined by $\alpha(Z)_{(i,m),(j,n)}=(Z_{ij})_{mn}$. Since $R$ contains $\mathbb M^{\fin}_{\N}(B)$, the algebra $\alpha(\mathbb M_4(R))$ contains $\mathbb M^{\fin}_{\{1,2,3,4\}\times\N}(B)$.
Relabelling $(r,n)$ as $4n+r-1$ finishes the proof.
\end{proof}

\begin{lemma}\label{prop:matrix-transfer}
Let $A$ be a $K$-algebra, $B$ a unital $K$-algebra, $\Lambda$ a nonempty set, and $\pi:A\to \mathbb M^{\rcf}_{\Lambda}(B)$ an algebra homomorphism whose image contains $\mathbb M^{\fin}_{\Lambda}(B)$. Set $Y=\ker\pi$ and $Z=\pi^{-1}(\mathbb M^{\fin}_{\Lambda}(B))$.
Then the map $\Phi:\Id(B)\to[Y,Z]$ defined by
\[
 \Phi(I)=\pi^{-1}(\mathbb M^{\fin}_{\Lambda}(I))
\]
is a lattice isomorphism with
$\Phi(IJ)=\Phi(I)\Phi(J)+Y$ for all $I,J\in \Id(B)$.
\end{lemma}
\begin{proof}
For any ideal $I\lhd B$ set $H_I=\mathbb M^{\fin}_{\Lambda}(I)$.
Clearly each $H_I$ is an ideal of $\im(\pi)$. Conversely, every ideal of $\im(\pi)$ contained in $H_B$ is an ideal of $H_B$, so matrix units identify it uniquely as some $H_I$.
It follows that the map $\Id(B)\to [0,H_B]_{\Id(\im(\pi))},~I\mapsto H_I$ is an order isomorphism. Clearly $H_IH_J=H_{IJ}$. Pullback under $\pi$ yields the claimed isomorphism and product formula.
\end{proof}

\begin{lemma}\label{lem:orbit}
Let $X$ be a set, $x_0\in X$ and $A$ a unital $K$-algebra. Let $S$ be a subalgebra of $\mathbb M^{\rcf}_{X}(A)$ that is generated by $AE_{x_0x_0}$ and a family $\mathcal G$ of partial permutation matrices closed under transpose. Let $\Omega$ be the set of all $x\in X$ for which there exist $n\geq 0$ and $G_1,\ldots,G_n\in\mathcal G$ such that $(G_1\cdots G_n)_{x_0,x}=1_A$, where the product for $n=0$ is the identity matrix. Then
\[
 SE_{x_0x_0}S=\mathbb M^{\fin}_{\Omega}(A).
\]
\end{lemma}
\begin{proof}
Set \(p=E_{x_0x_0}\). Let \(G\in\mathcal G\) and \(x\in\Omega\). Choose a product \(P\) of matrices in \(\mathcal G\) with \(P_{x_0,x}=1_A\). If \(G_{xy}\ne0\), then \((PG)_{x_0,y}=1_A\), so \(y\in\Omega\). Applying the same argument to \(G^T\) shows that \(G_{yx}\ne0\) also implies \(y\in\Omega\). Thus every generator of \(S\) has zero \((x,y)\)- and \((y,x)\)-entries whenever \(x\in\Omega\) and \(y\notin\Omega\). This property is preserved under sums and products, so it holds for every matrix in \(S\).

For \(s,t\in S\), we have
\[
(spt)_{xy}=s_{x,x_0}t_{x_0,y}.
\]
This entry is zero unless \(x,y\in\Omega\). Moreover, \(spt\) has only finitely many nonzero entries, since the \(x_0\)-column of \(s\) and the \(x_0\)-row of \(t\) have finite support. Hence \(SpS\subseteq\mathbb M^{\fin}_{\Omega}(A)\).

Conversely, let \(x,y\in\Omega\) and \(a\in A\). Choose products \(P,Q\) of matrices in \(\mathcal G\) such that \(P_{x_0,x}=Q_{x_0,y}=1_A\), replacing any empty product by \(p\). Then \(P^T,Q\in S\). Since \(P\) and \(Q\) are partial permutation matrices,
\[
E_{xy}(a)=P^T E_{x_0x_0}(a)Q
         =\bigl(P^T E_{x_0x_0}(a)\bigr)pQ\in SpS.
\]
Every matrix in \(\mathbb M^{\fin}_{\Omega}(A)\) is a finite sum of such matrices. Therefore \(SpS=\mathbb M^{\fin}_{\Omega}(A)\).
\end{proof}

\begin{lemma}\label{lem:finite-modifications}
Let $X$ be a set, $A$ a unital $K$-algebra and $\mathcal G\subseteq \mathbb M^{\rcf}_X(A)$ a family of partial-permutation matrices closed under transpose. Let $F\subseteq X$ be finite and $x_0\in F$. Modify some members of $\mathcal G$ by adding a matrix in $\mathbb M^{\fin}_F(A)$ or by multiplying on either side by $1+c$, where $c\in\mathbb M^{\fin}_F(A)$ and $(1+c)^{-1}=1+d$ with $d\in\mathbb M^{\fin}_F(A)$. Let $S$ be the subalgebra of $\mathbb M^{\rcf}_X(A)$ generated by the modified family. If
\[
 AE_{x_0x_0}\subseteq S,\qquad E_{xx_0},E_{x_0x}\in S\quad(x\in F),
\]
then $SE_{x_0x_0}S=\mathbb M^{\fin}_{\Omega}(A)$ for a nonempty subset $\Omega\subseteq X$.
\end{lemma}
\begin{proof}
The displayed matrices generate $\mathbb M^{\fin}_F(A)$, so every modification can be reversed and hence $\mathcal G\subseteq S$. Adjoin the matrices $E_{xx_0},E_{x_0x}~(x\in F)$ to $\mathcal G$, obtaining a transpose-closed family $\mathcal G'$. Then $S$ is generated by $\mathcal G'$ and $AE_{x_0x_0}$. The assertion of the lemma follows now from Lemma~\ref{lem:orbit}.
\end{proof}

\section{Representations of $L_K(E,w)$ in $\mathbb M^{\rcf}_X(A)$}
\subsection{A class of representations}
Let $(E,w)$ be a row-finite weighted graph and $A$ a unital $K$-algebra. For each vertex $v$ choose a countably infinite set $D^v$, with the sets $D^v$ pairwise disjoint, and set $X=\bigsqcup_{v\in E^0}D^v$ and $\mathcal M=\mathbb M^{\rcf}_X(A)$.

For any $v\in E^0$ let $P^v\in\mathcal M$ be the matrix with 
\[P^v_{xy}=\delta_{x,y}\delta_{x\in D^v}1_A\quad(x,y\in X).\]
Note that the $P^v$'s are pairwise orthogonal idempotents in $\mathcal M$.

For each $v\in E^0_{\reg}$ set $m_v=w(v)$ and write $s^{-1}(v)=\{e^{v,1},\dots,e^{v,n_v}\}$. Choose matrices $U^v\in\mathbb M_{m_v\times n_v}(\mathcal M)$ and $V^v\in\mathbb M_{n_v\times m_v}(\mathcal M)$ such that $U^v_{ij}=V^v_{ji}=0$ whenever $i>w(e^{v,j})$, and moreover
\begin{equation}\label{eq:defUV}
 L^vU^vQ^v=U^v,\quad Q^vV^vL^v=V^v,\quad U^vV^v=L^v,\quad V^vU^v=Q^v,
\end{equation}
where 
\begin{align*}
L^v&=\diag(P^v,\ldots,P^v)\in \mathbb M_{m_v\times m_v}(\mathcal M),\\
Q^v&=\diag(P^{r(e^{v,1})},\dots,P^{r(e^{v,n_v})})\in \mathbb M_{n_v\times n_v}(\mathcal M).
\end{align*}
It follows from the universal property of $L_K(E,w)$ that there is an algebra homomorphism $\rho:L_K(E,w)\to \mathcal M$ such that 
\[
 \rho(u)=P^u,\qquad
 \rho(e^{v,j}_i)=U^{v}_{ij},\qquad
 \rho((e^{v,j}_i)^*)=V^v_{ji}
\]
for all $u\in E^0$, $v\in E^0_{\reg}$, $1\leq j\leq n_v$ and $1\leq i\leq w(e^{v,j})$.

\subsection{Representations defined by branching systems}\label{subsec:branching-representations}
Let $D^v~(v\in E^0)$, $X$, $\mathcal M$ and $P^v~(v\in E^0)$ be as in \S 4.1. We choose the matrices $U^v, V^v~(v\in E^0_{\reg})$ as follows.

For each $v\in E^0_{\reg}$, $1\leq i\leq m_v$ and
$1\leq j\leq n_v$ choose subsets
\[
 D^v_{ij}\subseteq D^v,\qquad
 R^v_{ij}\subseteq D^{r(e^{v,j})}
\]
and a bijection
\[
 \theta^v_{ij}:D^v_{ij}\longrightarrow R^v_{ij}
\]
such that
\begin{equation}\label{eq:branching-partitions}
 \begin{aligned}
 D^v&=\bigsqcup_{j=1}^{n_v}
       D^v_{ij} &&\text{ for each }1\leq i\leq m_v,\\
 D^{r(e^{v,j})}&=\bigsqcup_{i=1}^{m_v}
       R^v_{ij} &&\text{ for each }1\leq j\leq n_v,\\
       D^v_{ij}&=R^v_{ij}=\emptyset&&\text{ if } i>w(e^{v,j}).
 \end{aligned}
\end{equation}
Together with $X$ and the vertex sets $D^v$, the sets $D^v_{ij}$ and
bijections $\theta^v_{ij}$ form an $(E,w)$-branching system in the sense of
\cite[Definition~11.3.1]{Survey}. Our $D^v,D^v_{ij},R^v_{ij},\theta^v_{ij}$
correspond to the survey's $D_v,R_{e_i},D_{e_i},g_{e_i}$, respectively,
where $e=e^{v,j}$.

Using the convention for partial permutation matrices from Section~3,
set
\[
 U^v_{ij}=P_{\theta^v_{ij}},\qquad
 V^v_{ji}=P_{(\theta^v_{ij})^{-1}}.
\]
Explicitly,
\[
 (U^v_{ij})_{xy}=(V^v_{ji})_{yx}=
 \begin{cases}
  1_A,&x\in D^v_{ij}\text{ and }\theta^v_{ij}(x)=y,\\
  0,&\text{otherwise}
 \end{cases}
 \qquad(x,y\in X).
\]
We leave it to the reader to check that the relations \eqref{eq:defUV} hold, and that $U^v_{ij}=V^v_{ji}=0$ whenever $i>w(e^{v,j})$. By \S 4.1, the matrices $P^v~(v\in E^0)$ and $U^v, V^v~(v\in E^0_{\reg})$ define a representation $\rho:L_K(E,w)\to \mathcal M$.

\medskip
\noindent\emph{Branching system with prescribed pairs.}
In Section~5 we will prescribe finitely many values
$\theta^v_{ij}(x)=y$ at each relevant vertex $v$. Every prescription must satisfy
$i\leq w(e^{v,j})$, $x\in D^v$ and $y\in D^{r(e^{v,j})}$.
Moreover, any two prescriptions $\theta^v_{ij}(x)=y$ and $\theta^v_{i'j'}(x')=y'$ at the same vertex $v$ must satisfy $(i,x)=(i',x')$ if and only if $(j,y)=(j',y')$. If these two conditions are satisfied, we say that the prescriptions are \textit{compatible}. We show that every family of compatible prescriptions that is finite at each vertex extends to a branching
system.

Fix $v\in E^0_{\reg}$. Let
$\theta^{v,0}_{ij}:D^{v,0}_{ij}\to R^{v,0}_{ij}$ be the finite
partial bijections prescribed at $v$, taking a map to be empty if
no value is prescribed. By compatibility, their domains are pairwise disjoint for each fixed $i$ as $j$ varies, and their ranges are pairwise disjoint for each fixed $j$ as $i$ varies. Set
\[
 \mathcal D_i=D^v\setminus\bigcup_{j=1}^{n_v}D^{v,0}_{ij},
 \qquad
 \mathcal R_j=D^{r(e^{v,j})}\setminus
                  \bigcup_{i=1}^{m_v}R^{v,0}_{ij}.
\]
All these sets are countably infinite. Choose $j_0$ with
$w(e^{v,j_0})=m_v$ and partition $\mathcal D_1$ into countably
infinite sets $T_1,\ldots,T_{n_v}$. For each $j\ne j_0$, extend
$\theta^{v,0}_{1j}$ by a bijection $T_j\to\mathcal R_j$.
Next partition $\mathcal R_{j_0}$ into countably infinite sets
$Z_1,\ldots,Z_{m_v}$. Extend $\theta^{v,0}_{1j_0}$ by a bijection
$T_{j_0}\to Z_1$, and for $2\leq i\leq m_v$ extend
$\theta^{v,0}_{ij_0}$ by a bijection $\mathcal D_i\to Z_i$.
Leave every other prescribed partial bijection unchanged. Taking the
domains and ranges of the resulting maps as $D^v_{ij}$ and
$R^v_{ij}$ yields \eqref{eq:branching-partitions}.
This construction adds no pairs with $i\geq2$ and $j\ne j_0$.

Performing the construction in the previous paragraph at every regular vertex yields a branching system satisfying all prescribed values. In particular, taking no prescribed values at any vertex proves that a branching system exists.

\section{Representations arising from failure of Condition (LPA)}
\begin{proposition}\label{prop:matrix-representation}
Let $(E,w)$ be a row-finite weighted graph that fails Condition (LPA), and set $L=L_K(E,w)$. Let $A$ be a nonzero unital $K$-algebra generated by elements $a,b$ with $b^4=1$. Then there exist a set $X$, a representation
\[
 \rho:L\longrightarrow\mathbb M^{\rcf}_X(A),
\]
and a nonempty subset $\Omega\subseteq X$ such that $\mathbb M^{\fin}_{\Omega}(A)$ is an ideal of $\rho(L)$.
\end{proposition}

\begin{proof}
Choose the sets $D^v~(v\in E^0)$ as in Section~4, and set $X=\bigsqcup_{v\in E^0}D^v$ and $\mathcal M=\mathbb M^{\rcf}_X(A)$. We consider the first of the conditions (LPA1)--(LPA4) that fails. In each case, the initial representation comes from a branching system as in \S\ref{subsec:branching-representations} (we prescribe finitely many values and then extend).

Thus the initial matrices $U^v_{ij}$ are partial permutation matrices, the $V^v_{ji}$ are their transposes, and the $P^v$ are diagonal partial permutation matrices. The family $\mathcal G$ consisting of these matrices is therefore closed under transpose. We then make modifications of the kind described in Lemma~\ref{lem:finite-modifications}, using matrices from $\mathbb M^{\fin}_F(A)$ where $F\subseteq X$ is a finite subset. Write $\rho$ for the modified representation and set $S=\rho(L)$. We show that $AE_{x_0x_0}\subseteq S$ and $E_{xx_0},E_{x_0x}\in S~(x\in F)$ for some $x_0\in F$. The assertion of the proposition then follows from Lemma~\ref{lem:finite-modifications}.

We use the same invertible two-by-two matrix in all four cases:
\begin{equation}\label{eq:heavyblock}
 M=\begin{pmatrix}b&b\\a&a-1\end{pmatrix},
 \qquad M^{-1}=\begin{pmatrix}(1-a)b^{-1}&1\\ab^{-1}&-1\end{pmatrix}.
\end{equation}
One checks easily that $MM^{-1}=M^{-1}M=I_2$ and
\[
 ME_{11}M^{-1}
 =\begin{pmatrix}b(1-a)b^{-1}&b\\a(1-a)b^{-1}&a\end{pmatrix}.
\]

\par\medskip
\Needspace{6\baselineskip}
\noindent\textbf{Case 1}\par\nopagebreak\smallskip
\noindent
Suppose (LPA1) fails. Then there is a vertex $u$ that emits two heavy edges. Choose distinct heavy edges $e,f\in s^{-1}(u)$ such that $w(f)=w(u)$ (i.e. $f$ has maximal weight among the edges emitted by $u$). Write
\[
 e=e^{u,j_e},\qquad f=e^{u,j_f}
\]
using the edge indexing of Section~4. Choose pairwise distinct
\[
 x_0\in D^u,\qquad z_e\in D^{r(e)},\qquad z_f\in D^{r(f)}.
\]

For the initial branching system, prescribe
\[
 \theta^u_{1j_e}(x_0)=z_e,\qquad
 \theta^u_{2j_f}(x_0)=z_f.
\]
Complete the branching system by the extension construction in
\S\ref{subsec:branching-representations}, choosing $j_f$ as the
maximum-weight index at $u$. Write $U^{u,0}$ and $V^{u,0}$ for the initial
matrices, and set $p=E_{x_0x_0}$. The prescribed values and the
branching partitions imply
\[
\begin{aligned}
pU^{u,0}_{1j}&=\delta_{j,j_e}E_{x_0z_e},\quad&
pU^{u,0}_{2j}&=\delta_{j,j_f}E_{x_0z_f},\\
V^{u,0}_{j1}p&=\delta_{j,j_e}E_{z_ex_0},&
V^{u,0}_{j2}p&=\delta_{j,j_f}E_{z_fx_0},
\end{aligned}
\qquad(1\leq j\leq n_u).
\]

Let $W\in\mathbb M_{m_u}(\mathcal M)$ be the identity matrix except
for its upper-left $2\times2$ block, which is
\[
 \begin{pmatrix}
  1-p+E_{x_0x_0}(b)&E_{x_0x_0}(b)\\
  E_{x_0x_0}(a)&1-p+E_{x_0x_0}(a-1)
 \end{pmatrix}.
\]
Since $M$ is invertible, $W$ is invertible. Its inverse is the
identity matrix except for its upper-left $2\times2$ block, which is
\[
 \begin{pmatrix}
  1-p+E_{x_0x_0}((1-a)b^{-1})&p\\
  E_{x_0x_0}(ab^{-1})&1-2p
 \end{pmatrix}.
\]
Replace $U^{u,0},V^{u,0}$ by
\[
 U^u=WU^{u,0},\qquad V^u=V^{u,0}W^{-1},
\]
leaving the matrices at all other vertices unchanged.
Since $W$ commutes with $L^u$, we have
\[
 U^uV^u=WL^uW^{-1}=L^u,\qquad
 V^uU^u=V^{u,0}U^{u,0}=Q^u.
\]
The support relations in \eqref{eq:defUV} are also preserved.
By the displayed identities for $pU^{u,0}_{1j}$ and
$pU^{u,0}_{2j}$, only the entries at $(x_0,z_e)$ in
$U^{u,0}_{1j_e},U^{u,0}_{2j_e}$ and at $(x_0,z_f)$ in
$U^{u,0}_{1j_f},U^{u,0}_{2j_f}$ can change. Similarly, in $V^{u,0}$ only entries at the
corresponding transposed positions can change.
It follows that $U^u_{ij}=V^u_{ji}=0$ whenever $i>w(e^{u,j})$. Hence the modified family defines a representation
$\rho:L\to\mathbb M^{\rcf}_X(A)$, whose image we denote by $S$.

Next we show that $Ap\subseteq S$. The extension construction adds
no pairs to $\theta^u_{ij_e}$ for $i\geq2$. Hence $D^u_{ij_e}=R^u_{ij_e}=\emptyset$ for
$2\leq i\leq w(e)$, and $R^u_{1j_e}=D^{r(e)}$.
It follows that, after modification,
\[
 U^u_{2j_e}=E_{x_0z_e}(a),\qquad
 V^u_{j_e2}=E_{z_ex_0}.
\]
Moreover,
\[
 (U^u_{1j_e})_{xz_e}=\delta_{x,x_0}b\quad(x\in X),
\]
and hence
\[
 U^u_{1j_e}V^u_{j_e2}=bp,\qquad
 U^u_{2j_e}V^u_{j_e2}=ap.
\]
Thus $ap,bp\in S$, and $p=(bp)^4\in S$ because $b^4=1_A$.
Since $a,b$ generate $A$, we obtain $Ap\subseteq S$.

Finally, we show that $E_{xx_0},E_{x_0x}\in S$ for all
$x\in F=\{x_0,z_e,z_f\}$. Since $b^4=1_A$,
\[
 b^{-1}p=(bp)^3\in S.
\]
The formulas defining the modification imply
\[
 \begin{aligned}
 (U^u_{1j_e})_{x_0,x}&=\delta_{x,z_e}b,&
 (U^u_{1j_f})_{x_0,x}&=\delta_{x,z_f}b,\\
 (V^u_{j_e2})_{x,x_0}&=\delta_{x,z_e}1_A,&
 (V^u_{j_f2})_{x,x_0}&=-\delta_{x,z_f}1_A
 \end{aligned}
 \quad(x\in X).
\]
Consequently,
\[
 \begin{aligned}
 E_{x_0z_e}&=(b^{-1}p)U^u_{1j_e}\in S,&
 E_{z_ex_0}&=V^u_{j_e2}p\in S,\\
 E_{x_0z_f}&=(b^{-1}p)U^u_{1j_f}\in S,&
 E_{z_fx_0}&=-V^u_{j_f2}p\in S.
 \end{aligned}
\]
Together with $p\in S$, this proves that
$E_{xx_0},E_{x_0x}\in S$ for every $x\in F$.

Each modified matrix is obtained from the corresponding initial partial permutation matrix by adding a matrix supported on \(F\times F\).
All hypotheses of Lemma~\ref{lem:finite-modifications} are therefore
satisfied. It follows that $\mathbb M^{\fin}_{\Omega}(A)\lhd S$ for a nonempty set $\Omega\subseteq X$, as required.\\
\\
{\bf A common extraction calculation}\\
Let $x_1,x_2\in X$ be distinct and set $F=\{x_1,x_2\}$.
The last three cases will use matrices $C_{12},C_{21},C_{22},R\in S$
with the following properties:
\begin{enumerate}[(i)]
\item $C_{22}$ is diagonal with entries in $\{0,1_A\}$,
$(C_{22})_{x_1,x_1}=0$, $(C_{22})_{x_2,x_2}=1_A$, and
$C_{21}C_{22}=0$.
\item For every $x\in X$,
\[
 (C_{21})_{x,x_1}=\delta_{x,x_2}1_A,\qquad
 (C_{21})_{x_2,x}=\delta_{x,x_1}1_A,\qquad
 (C_{12})_{x,x_2}=\delta_{x,x_1}1_A.
\]
\item The entries of $R$ on the ordered coordinates $(x_1,x_2)$ are
\[
 \begin{pmatrix}R_{x_1,x_1}&R_{x_1,x_2}\\
 R_{x_2,x_1}&R_{x_2,x_2}\end{pmatrix}
 =\begin{pmatrix}b(1-a)b^{-1}&b\\a(1-a)b^{-1}&a\end{pmatrix}.
\]
If at least one of $x,y$ lies outside $F$ and $x\ne y$, then $R_{xy}=0$.
\end{enumerate}
Let $R_0$ have the same diagonal entries as $R$ outside $F$ and
zero entries everywhere else. Since $C_{22}$ is diagonal with
entries $0$ or $1_A$, it commutes with $R_0$. Thus
\[
 C_{21}R_0C_{22}=C_{21}C_{22}R_0=0.
\]
Moreover,
\[
 RC_{22}=R_0C_{22}+E_{x_1x_2}(b)+E_{x_2x_2}(a).
\]
Clearly 
$C_{21}E_{x_1x_2}(b)=E_{x_2x_2}(b)$ (note that $C_{21}C_{22}=0$ implies that the $x_2$-column of $C_{21}$ is zero). Consequently,
\begin{equation}\label{eq:rank-one-extraction}
 \xi=C_{21}RC_{22}=E_{x_2x_2}(b),\qquad
 p=\xi^4=E_{x_2x_2}\in S,\qquad pRp=E_{x_2x_2}(a).
\end{equation}
Since $a,b$ generate $A$, it follows that $Ap\subseteq S$.
Finally, the entries in (ii) yield
\[
 C_{12}p=E_{x_1x_2}\in S,\qquad
 pC_{21}=E_{x_2x_1}\in S.
\]
These are the connecting matrix units needed for
Lemma~\ref{lem:finite-modifications} with $F=\{x_1,x_2\}$ and
$x_0=x_2$.
\par\medskip
\Needspace{6\baselineskip}
\noindent\textbf{Case 2}\par\nopagebreak\smallskip
\noindent
Suppose (LPA1) holds and (LPA2) fails. Then we can choose a heavy edge
$e=e^{u,j_e}$ and a vertex $v\in T(r(e))$ emitting at least two
edges, with $v$ at minimum distance from $r(e)$ among such vertices.
Let $h=k^{(1)}\cdots k^{(\ell)}$ be the shortest path from $r(e)$
to $v$. Each vertex of $h$ before $v$ emits exactly one edge, and
$h$ has no repeated vertices. Also, $u$ cannot occur before the
last vertex of $h$: its unique emitted edge would then be $e$,
forcing a repetition of $r(e)$. Thus $u$ either lies outside $h$
or equals $v$. Choose distinct edges
$f=e^{v,j_f}$ and $g=e^{v,j_g}$.

We first construct a branching system as in Section~4. Take distinct
$x_1,x_2\in D^v$ and set $F=\{x_1,x_2\}$. Write
$k^{(q)}=e^{v_q,1}$ for $1\leq q\leq \ell$, and set $v_{\ell+1}=v$.
Choose distinct coordinates
$x_1^{(q)},x_2^{(q)}\in D^{v_q}$ for $1\leq q\leq\ell$,
and set $x_i^{(\ell+1)}=x_i$ for $i=1,2$. Prescribe
\[
 \theta^{v_q}_{11}(x_i^{(q)})=x_i^{(q+1)}
 \quad(1\leq q\leq\ell,\ i=1,2).
\]
Set $y_i=x_i^{(1)}~(i=1,2)$. If $\ell=0$, set $y_i=x_i~(i=1,2)$ instead.
Choose $t\in D^u$, with $t\notin F$ if $u=v$, and prescribe
\[
 \theta^u_{1j_e}(t)=y_1,\qquad
 \theta^u_{2j_e}(t)=y_2.
\]
Finally, choose $z_f\in D^{r(f)}$ and $z_g\in D^{r(g)}$. If $f = e$, choose $z_f\neq y_1,y_2$; if $g = e$, choose $z_g\neq y_1,y_2$. Prescribe
\[
 \theta^v_{1j_g}(x_1)=z_g,\qquad
 \theta^v_{1j_f}(x_2)=z_f.
\]
By \S\ref{subsec:branching-representations} we can extend these prescriptions to a branching system.

Define the following products of the initial matrices:
\[
 H=U^{v_1}_{11}\cdots U^{v_\ell}_{11},\qquad
 J=V^{v_\ell}_{11}\cdots V^{v_1}_{1 1}.
\]
For $\ell=0$ take $H=J=P^v$. The matrices $H,J$ are partial
permutation matrices which are transposes of one another, and
\[
 H_{y_i,x_i}=J_{x_i,y_i}=1_A\quad(i=1,2).
\]
Since each source along $h$ emits only one edge, relations
\eqref{eq:defUV} yield $HJ=P^{r(e)}$.
Also, the initial matrix $R^{(0)}=U^v_{1j_g}V^v_{j_g1}$ is diagonal
with entries $0$ or $1_A$, and the last two prescribed matches yield
\[
 R^{(0)}_{x_1,x_1}=1_A,\qquad R^{(0)}_{x_2,x_2}=0.
\]

Let $W\in\mathcal M$ be defined by
\[
 \begin{pmatrix}W_{x_1,x_1}&W_{x_1,x_2}\\
 W_{x_2,x_1}&W_{x_2,x_2}\end{pmatrix}=M,
 \qquad W_{xy}=\delta_{xy}1_A
 \quad\text{if }(x,y)\notin F\times F.
\]
Its inverse is defined in the same way, but using $M^{-1}$ instead of $M$.
For every $1\leq j\leq n_v$, replace
\[
 U^v_{1j}\longmapsto WU^v_{1j},\qquad
 V^v_{j1}\longmapsto V^v_{j1}W^{-1},
\]
leaving all other matrices unchanged. We leave it to the reader to check that the modified family of matrices still defines a representation $\rho:L\to \mathbb M^{\rcf}_X(A)$, whose image we denote by $S$. Below, $U$ and $V$
mean the modified matrices.

Set
\[
 C_{ij}=JV^u_{j_e i}U^u_{j j_e}H\quad(1\leq i,j\leq2),
 \qquad R=U^v_{1j_g}V^v_{j_g1}.
\]
Initially, $V^u_{j_e1}U^u_{1j_e}$ and
$V^u_{j_e2}U^u_{2j_e}$ are diagonal projections with disjoint
supports: the sets $R^u_{1j_e}$ and $R^u_{2j_e}$ in the branching system
are disjoint. These two products are unchanged by the
modification, also when $u=v$, because the additional factors in
the first product cancel. It follows that $C_{11},C_{22}$ are
still diagonal projections with disjoint supports, and the prescribed
matches yield
\[
 (C_{11})_{x_1,x_1}=(C_{22})_{x_2,x_2}=1_A,\qquad
 (C_{22})_{x_1,x_1}=0.
\]
Since $e$ is the only heavy edge emitted by $u$, the $(1,2)$-entry
of $U^uV^u=L^u$ is
$U^u_{1j_e}V^u_{j_e2}=0$. Therefore
\[
 \begin{aligned}
 C_{21}C_{22}
 &=JV^u_{j_e2}U^u_{1j_e}HJ V^u_{j_e2}U^u_{2j_e}H\\
 &=JV^u_{j_e2}(U^u_{1j_e}V^u_{j_e2})U^u_{2j_e}H=0.
 \end{aligned}
\]

Before modification, the prescribed matches yield entries $1_A$
at $(x_2,x_1)$ in $C_{21}$ and at $(x_1,x_2)$ in $C_{12}$.
Each is the only nonzero entry in its row and column, since these
products are initially partial permutation matrices. If $u\ne v$,
they do not change. Suppose $u=v$. Write $U^{u,0}$ and $V^{u,0}$ for the initial
matrices and set
\[
 A_i=U^{u,0}_{i j_e}H,\qquad
 B_i=JV^{u,0}_{j_e i}\quad(i=1,2).
\]
Then the modified products are
\[
 C_{21}=B_2WA_1,\qquad C_{12}=B_1W^{-1}A_2.
\]
For $x\in X$, let $\mathbf e_x$ denote the column vector with
entry $1_A$ at $x$ and zeros elsewhere.
The prescribed values imply $(A_i)_{t,x_i}=1_A$.
Since $A_i$ is a partial permutation matrix and $B_i=A_i^T$,
we have
\[
 A_i\mathbf e_{x_i}=\mathbf e_t,\qquad
 B_i\mathbf e_t=\mathbf e_{x_i},\qquad
 \mathbf e_{x_i}^TB_i=\mathbf e_t^T,\qquad
 \mathbf e_t^TA_i=\mathbf e_{x_i}^T
 \quad(i=1,2).
\]
Moreover, $t\notin F$ implies
\[
 W\mathbf e_t=W^{-1}\mathbf e_t=\mathbf e_t,
 \qquad \mathbf e_t^TW=\mathbf e_t^T.
\]
Consequently,
\begin{align*}
 C_{21}\mathbf e_{x_1}
 &=B_2WA_1\mathbf e_{x_1}
  =B_2W\mathbf e_t=B_2\mathbf e_t=\mathbf e_{x_2},\\
 \mathbf e_{x_2}^TC_{21}
 &=\mathbf e_{x_2}^TB_2WA_1
  =\mathbf e_t^TWA_1=\mathbf e_t^TA_1=\mathbf e_{x_1}^T,\\
 C_{12}\mathbf e_{x_2}
 &=B_1W^{-1}A_2\mathbf e_{x_2}
  =B_1W^{-1}\mathbf e_t=B_1\mathbf e_t=\mathbf e_{x_1}.
\end{align*}
Hence
\[
 (C_{21})_{x,x_1}=\delta_{x,x_2}1_A,\qquad
 (C_{21})_{x_2,x}=\delta_{x,x_1}1_A,\qquad
 (C_{12})_{x,x_2}=\delta_{x,x_1}1_A
\]
for all $x\in X$. Finally, $R=WR^{(0)}W^{-1}$. Since $R^{(0)}$ is diagonal and
has diagonal entries $1_A,0$ at $x_1,x_2$, its four resulting
entries on $F\times F$ are exactly those in (iii) of the common
extraction calculation. Its other off-diagonal entries are zero.
The common extraction calculation now yields
\[
 Ap\subseteq S,\qquad C_{12}p=E_{x_1x_2}\in S,\qquad
 pC_{21}=E_{x_2x_1}\in S,\qquad p=E_{x_2x_2}.
\]
All modified generators were obtained by multiplication by $W$ or
$W^{-1}$, whose differences from the identity are supported on
$F\times F$. Thus Lemma~\ref{lem:finite-modifications} applies with
$x_0=x_2$. \\
\\
{\bf Path matrix units for the last two cases}\\
Suppose (LPA1) and (LPA2) hold. Let $k^{(1)}\cdots k^{(n)}$ be a path
with distinct vertices, beginning with a heavy edge
$k^{(1)}=e=e^{u,j_e}$ and ending at $z$.
Set $v_q=s(k^{(q)})$ for $1\leq q\leq n$ (so $v_1=u$), and $v_{n+1}=z$.
For $2\leq q\leq n$, every $v_q$ emits exactly one edge, and hence
$k^{(q)}=e^{v_q,1}$. Let $U^v,V^v~(v\in E^0_{\reg})$ be matrices satisfying the conditions of \S4.1. Define
\[
 \begin{aligned}
 H&=U^{v_2}_{11}\cdots U^{v_n}_{11},&
 J&=V^{v_n}_{11}\cdots V^{v_2}_{11},\\
 B_i&=U^u_{i j_e}H,&
 G_i&=JV^u_{j_e i}\quad(i=1,2).
 \end{aligned}
\]
For $n=1$ set $H=J=P^{r(e)}$.
Relations \eqref{eq:defUV} yield $HJ=P^{r(e)}$ and
\[
 B_1G_2=B_2G_1=0,\qquad B_2G_2=P^u,\qquad
 Q:=B_1G_1=P^u-\sum_{j\ne j_e}U^u_{1j}V^u_{j1}
\]
(note that $e$ is the only heavy edge emitted by $u$ since (LPA1) holds). For $j,k\ne j_e$,
relation $V^uU^u=Q^u$ implies
\[
 V^u_{j1}U^u_{1k}=\delta_{jk}P^{r(e^{u,j})}.
\]
Thus the summands in the expression for $Q$ are pairwise orthogonal
idempotents in $P^u\mathcal M P^u$, and $Q^2=Q$.
Set
\[
 A_{ij}=G_iQB_j\quad(1\leq i,j\leq2).
\]
Since
\[
 Q(B_jG_k)Q=
 \begin{cases}Q,&j=k,\\0,&j\ne k,\end{cases}
\]
we obtain
\[
 A_{ij}A_{kl}=G_iQ(B_jG_k)QB_l=\delta_{jk}A_{il}.
\]

We choose the matrices $U^v,V^v~(v\in E^0_{\reg})$ as follows. Take distinct
$x_1,x_2\in D^z$ and $t\in D^u$. For $2\leq q\leq n$, choose
distinct $x_1^{(q)},x_2^{(q)}\in D^{v_q}$, and set
$x_i^{(n+1)}=x_i$ for $i=1,2$. Prescribe
\[
 \theta^u_{1j_e}(t)=x_1^{(2)},\qquad
 \theta^u_{2j_e}(t)=x_2^{(2)},
\]
and
\[
 \theta^{v_q}_{11}(x_i^{(q)})=x_i^{(q+1)}
 \quad(2\leq q\leq n,\ i=1,2).
\]
The vertices of the path are distinct, so these prescriptions are
compatible and extend to a branching system by
\S\ref{subsec:branching-representations}.
They imply
\[
 (B_i)_{t,x_i}=(G_i)_{x_i,t}=1_A\quad(i=1,2),\qquad Q_{tt}=1_A.
\]
Initially, $B_i$ is a partial permutation matrix, $G_i=B_i^T$,
and $Q=B_1B_1^T$ is diagonal with entries $0$ or $1_A$.
Hence $A_{ij}$ is a partial permutation matrix, $A_{12}=A_{21}^T$,
and $A_{11},A_{22}$ are diagonal with entries $0$ or $1_A$.
Using the column vectors $\mathbf e_x$ from Case~2, we have
\[
 B_i\mathbf e_{x_i}=\mathbf e_t,\qquad
 G_i\mathbf e_t=\mathbf e_{x_i},\qquad
 \mathbf e_{x_i}^TG_i=\mathbf e_t^T,\qquad
 \mathbf e_t^TB_i=\mathbf e_{x_i}^T,
\]
and $Q\mathbf e_t=\mathbf e_t$, $\mathbf e_t^TQ=\mathbf e_t^T$.
Consequently,
\[
 A_{ij}\mathbf e_{x_j}=G_iQ\mathbf e_t=\mathbf e_{x_i},\qquad
 \mathbf e_{x_i}^TA_{ij}=\mathbf e_t^TQB_j=\mathbf e_{x_j}^T.
\]
These equalities also determine the remaining rows and columns
indexed by $F=\{x_1,x_2\}$. If $r\ne j$, then
\[
 A_{ij}\mathbf e_{x_r}=A_{ij}A_{rr}\mathbf e_{x_r}=0,
\]
and, if $r\ne i$, then
\[
 \mathbf e_{x_r}^TA_{ij}=\mathbf e_{x_r}^TA_{rr}A_{ij}=0.
\]
Thus, for $1\leq i,j,r\leq2$ and every $x\in X$,
\begin{equation}\label{eq:path-reserved-entries}
 (A_{ij})_{x,x_r}=\delta_{jr}\delta_{x,x_i}1_A,\qquad
 (A_{ij})_{x_r,x}=\delta_{ri}\delta_{x,x_j}1_A.
\end{equation}
In particular, the submatrix of $A_{ij}$ on the ordered coordinates
$(x_1,x_2)$ is the matrix unit $E_{ij}$, and
$(A_{ij})_{xy}=0$ whenever exactly one of $x,y$ belongs to $F$.

We will modify the matrices at the final edge
$k^{(n)}=e^{v_n,j}$, where $j=j_e$ if $n=1$ and $j=1$ otherwise.
Let $W\in\mathcal M$ be invertible, with
$W_{xy}=\delta_{xy}1_A$ whenever $(x,y)\notin D^z\times D^z$.
For $1\leq i\leq w(k^{(n)})$, replace
\[
 U^{v_n}_{ij}\longmapsto U^{v_n}_{ij}W^{-1},\qquad
 V^{v_n}_{ji}\longmapsto WV^{v_n}_{ji},
\]
leaving all other matrices unchanged. If $W-1$ and $W^{-1}-1$ are supported on $F\times F$, these are
modifications of the kind allowed in Lemma~\ref{lem:finite-modifications}. We leave it to the reader to check that the modified family of matrices still defines a representation.
Set
\[
 B_i'=B_iW^{-1},\qquad G_i'=WG_i,\qquad
 Q'=B_1'G_1'=Q,\qquad A_{ij}'=WA_{ij}W^{-1}.
\]
These are the matrices obtained from the modified representation by the definitions above.

\par\medskip
\Needspace{6\baselineskip}
\noindent\textbf{Case 3}\par\nopagebreak\smallskip
\noindent
Suppose (LPA1) and (LPA2) hold and (LPA3) fails. Then we can choose heavy edges $e,f$ that are not in line, such that $T(r(e))\cap T(r(f))\neq \emptyset$. Choose paths $\gamma,\eta$ with minimum total length such that $\gamma$ starts with $e$, $\eta$ starts with $f$, and both $\gamma$ and $\eta$ end at a vertex $z\in T(r(e))\cap T(r(f))$. The paths $\gamma,\eta$ meet only at $z$, since an earlier common vertex would yield shorter paths. Neither initial source occurs on the other path,
since the edges are not in line. A repeated noninitial vertex would
also yield a shorter path. If either path returned to its initial
source, (LPA2) would force all subsequent edges to stay on the
resulting cycle. Then $z$ would lie on that cycle, so the other path
could continue from $z$ to the first source, contrary to the fact
that $e,f$ are not in line. Thus both paths have distinct vertices,
and every source after the first emits exactly one edge.

Choose distinct $x_1,x_2\in D^z$ and set $F=\{x_1,x_2\}$.
Prescribe the branching maps along each path as in the preceding
construction, using the same terminal coordinates $x_1,x_2$.
The two paths have disjoint sets of sources, so these prescriptions
are compatible. Extend them to a branching system by
\S\ref{subsec:branching-representations}.
Denote the two systems of path matrix units by
$\alpha_{ij}$ and $\beta_{ij}$, respectively.
Both satisfy \eqref{eq:path-reserved-entries} for the same $x_1,x_2$.

Let $W\in\mathcal M$ be defined by
\[
 \begin{pmatrix}W_{x_1,x_1}&W_{x_1,x_2}\\
 W_{x_2,x_1}&W_{x_2,x_2}\end{pmatrix}=M,
 \qquad W_{xy}=\delta_{xy}1_A
 \quad\text{if }(x,y)\notin F\times F.
\]
Its inverse is defined in the same way, but using $M^{-1}$ instead of $M$.
Write the final edge of $\eta$ as $k=e^{q,j_k}$;
here $j_k=1$ unless $\eta$ consists only of $f$.
Replace
\[
 U^q_{i j_k}\longmapsto U^q_{i j_k}W^{-1},\qquad
 V^q_{j_k i}\longmapsto WV^q_{j_k i}
 \quad(1\leq i\leq w(k)),
\]
leaving all other matrices unchanged.
The modified family of matrices defines a representation
$\rho:L\to\mathcal M$, whose image we denote by $S$.
Since the paths have disjoint sets of sources, the first system of path matrix units $\alpha_{ij}$ is
unchanged, while the second becomes
$\beta_{ij}'=W\beta_{ij}W^{-1}$. In particular,
$\alpha_{ij},\beta_{ij}'\in S$.

Set $C_{ij}=\alpha_{ij}$ and $R=\beta_{11}'$.
The matrix $C_{22}$ is diagonal with entries $0$ or $1_A$, and
\[
 (C_{22})_{x_1,x_1}=0,\qquad (C_{22})_{x_2,x_2}=1_A,
 \qquad C_{21}C_{22}=0.
\]
Moreover, \eqref{eq:path-reserved-entries} yields
\[
 (C_{21})_{x,x_1}=\delta_{x,x_2}1_A,\qquad
 (C_{21})_{x_2,x}=\delta_{x,x_1}1_A,\qquad
 (C_{12})_{x,x_2}=\delta_{x,x_1}1_A
 \quad(x\in X).
\]
Since $\beta_{11}$ is diagonal with diagonal entries $1_A,0$ at
$x_1,x_2$, the four entries of $R=W\beta_{11}W^{-1}$ on
$F\times F$ are
\[
 \begin{pmatrix}R_{x_1,x_1}&R_{x_1,x_2}\\
 R_{x_2,x_1}&R_{x_2,x_2}\end{pmatrix}
 =\begin{pmatrix}b(1-a)b^{-1}&b\\a(1-a)b^{-1}&a\end{pmatrix},
\]
and all other off-diagonal entries of $R$ are zero.
The common extraction calculation therefore yields
\[
 Ap\subseteq S,\qquad \alpha_{12}p=E_{x_1x_2}\in S,\qquad
 p\alpha_{21}=E_{x_2x_1}\in S
\]
where $p=E_{x_2x_2}$.
Since $W-1$ and $W^{-1}-1$ are supported on $F\times F$,
Lemma~\ref{lem:finite-modifications} applies with $x_0=x_2$.
\par\medskip
\Needspace{6\baselineskip}
\noindent\textbf{Case 4}\par\nopagebreak\smallskip
\noindent
Suppose (LPA1)--(LPA3) hold and (LPA4) fails. Then we can choose a heavy edge
$e$ and a cycle reachable from $r(e)$ which does not contain $e$.
By (LPA2), every vertex reachable from $r(e)$ emits at most one edge.
Choose the path beginning with $e$ and ending at its first vertex
$z$ on the cycle. It has no repeated noninitial vertex: after such a
repetition, its edges would stay on a cycle and could not first reach
$z$ later. Its initial source also lies outside both the cycle and
the rest of the path. If it lay on the cycle, (LPA2) would force $e$
to be the cycle edge emitted there, contrary to our choice of cycle.
If the path returned to its initial source before reaching $z$,
its uniquely determined continuation would repeat that closed path
and could not first reach $z$ later.
Thus the incoming path has distinct vertices, and every source after
the first emits exactly one edge.

Choose distinct $x_1,x_2\in D^z$ and set $F=\{x_1,x_2\}$.
Prescribe the branching maps along the incoming path as in the
path construction above. Write the cycle starting at $z$ as
$k^{(1)}\cdots k^{(n)}$, and set $v_q=s(k^{(q)})$ for
$1\leq q\leq n$ and $v_{n+1}=z$. Then
$k^{(q)}=e^{v_q,1}$ and $v_1=z=v_{n+1}$.
For $2\leq q\leq n$, choose distinct
$x_1^{(q)},x_2^{(q)}\in D^{v_q}$, and set
$x_i^{(1)}=x_i^{(n+1)}=x_i$ for $i=1,2$. Prescribe
\[
 \theta^{v_q}_{11}(x_i^{(q)})=x_i^{(q+1)}
 \quad(1\leq q\leq n,\ i=1,2).
\]
For $n=1$, these are simply
$\theta^z_{11}(x_i)=x_i$ for $i=1,2$.
The incoming path has no source on the cycle, so all prescriptions
are compatible. Extend them to a branching system by
\S\ref{subsec:branching-representations}.
Let $A_{ij}$ be the incoming-path matrix units. They satisfy
\eqref{eq:path-reserved-entries}.

Define 
\[
 Y^0=U^{v_1}_{11}\cdots U^{v_n}_{11},\qquad
 Z_0=V^{v_n}_{11}\cdots V^{v_1}_{11}.
\]
The prescribed values imply $(Y^0)_{x_i,x_i}=1_A$ for $i=1,2$.
Since $Y^0$ is a partial permutation matrix and $Z_0=(Y^0)^T$, we have
\[
 Y^0\mathbf e_{x_i}=Z_0\mathbf e_{x_i}=\mathbf e_{x_i},\qquad
 \mathbf e_{x_i}^TY^0=\mathbf e_{x_i}^TZ_0=\mathbf e_{x_i}^T
 \quad(i=1,2).
\]
Equivalently, for every $x\in X$,
\[
 (Y^0)_{x,x_i}=(Z_0)_{x,x_i}=\delta_{x,x_i}1_A,\qquad
 (Y^0)_{x_i,x}=(Z_0)_{x_i,x}=\delta_{x,x_i}1_A.
\]

Define $W$ as in Case~3, with submatrix $M$ on the ordered
coordinates $(x_1,x_2)$ and the identity outside $F\times F$.
At the final cycle edge replace
\[
 U^{v_n}_{i1}\longmapsto U^{v_n}_{i1}W,\qquad
 V^{v_n}_{1i}\longmapsto W^{-1}V^{v_n}_{1i}
 \quad(1\leq i\leq w(k^{(n)})),
\]
leaving all other matrices unchanged. The modified family of matrices defines a representation $\rho:L\to\mathcal M$, whose image we denote by $S$.
The incoming-path matrices $A_{ij}$ are unchanged.
The products of the modified cycle matrices are
\[
 Y=Y^0W\in S,\qquad Z=W^{-1}Z_0\in S.
\]

Set $R=YA_{11}Z$ and $R_0=Y^0A_{11}Z_0$.
The matrix $R_0$ is diagonal with entries $0$ or $1_A$, because
$A_{11}$ is diagonal and $Z_0=(Y^0)^T$.
The displayed rows and columns of $Y^0$ and $Z_0$ imply
\[
 (R_0)_{x_1,x_1}=1_A,\qquad (R_0)_{x_2,x_2}=0,
\]
and
\[
 Y^0(W-1)=(W-1)Y^0=W-1,\qquad
 Z_0(W-1)=(W-1)Z_0=W-1.
\]
Thus $Y^0$ and $Z_0$ commute with $W$ and $W^{-1}$, and
\[
 R=Y^0WA_{11}W^{-1}Z_0=WR_0W^{-1}.
\]
Consequently,
\[
 \begin{pmatrix}R_{x_1,x_1}&R_{x_1,x_2}\\
 R_{x_2,x_1}&R_{x_2,x_2}\end{pmatrix}
 =\begin{pmatrix}b(1-a)b^{-1}&b\\a(1-a)b^{-1}&a\end{pmatrix},
\]
and all other off-diagonal entries of $R$ are zero.
The matrices $C_{ij}=A_{ij}$ satisfy the remaining hypotheses of
the common extraction calculation by their matrix-unit relations
and \eqref{eq:path-reserved-entries}, exactly as in Case~3.
It follows that
\[
 Ap\subseteq S,\qquad A_{12}p=E_{x_1x_2}\in S,\qquad
 pA_{21}=E_{x_2x_1}\in S
\]
where $p=E_{x_2x_2}$. Since $W-1$ and $W^{-1}-1$ are supported on $F\times F$,
Lemma~\ref{lem:finite-modifications} applies with $x_0=x_2$.

\medskip
In each case, Lemma~\ref{lem:finite-modifications} yields a nonempty subset \(\Omega\subseteq X\) such that \(\mathbb M^{\fin}_{\Omega}(A)\) is an ideal of \(S=\rho(L)\). This proves the proposition.
\end{proof}

\section{The main theorem}
\begin{theorem}\label{thm:main}
Let $(E,w)$ be a row-finite weighted graph. The following are equivalent
\begin{enumerate}[(T1)]
\medskip
\item $(E,w)$ fails Condition (LPA).
\medskip
\item $\Id(L_K(E,w))$ is nondistributive.
\medskip
\item Ideal multiplication in $L_K(E,w)$ is not commutative.
\medskip
\item For every countably generated unital $K$-algebra $B$ there exist ideals $Y\subseteq Z$ of $L_K(E,w)$ and a lattice isomorphism
\[
 \Phi:\Id(B)\longrightarrow[Y,Z]
\]
which also satisfies, for all $I,J\lhd B$,
\begin{equation}\label{eq:relativeproduct}
 \Phi(IJ)=\Phi(I)\Phi(J)+Y.
\end{equation}
\end{enumerate}
\end{theorem}
\begin{proof}
Let $L=L_K(E,w)$. We first prove (T1)$\Rightarrow$(T4). Assume that
$(E,w)$ fails Condition (LPA), and let $B$ be a nonzero countably
generated unital $K$-algebra. By Lemma~\ref{lem:uniform-coefficients},
there are $a,b\in\mathbb M^{\rcf}_{\N}(B)$ such that $b^4=1$
and the subalgebra $A$ of $\mathbb M^{\rcf}_{\N}(B)$ generated by $a,b$ contains $\mathbb M^{\fin}_{\N}(B)$. In particular, $A$ is nonzero and unital, since $1=b^4\in A$.
By Proposition~\ref{prop:matrix-representation}, there exist a representation \(\rho:L\to\mathbb M^{\rcf}_X(A)\) and a nonempty subset \(\Omega\subseteq X\) such that
\[
Q:=\mathbb M^{\fin}_{\Omega}(A)\lhd S:=\rho(L).
\] For $s\in S$ and $\omega\in\Omega$, both $sE_{\omega\omega}$ and
$E_{\omega\omega}s$ belong to $Q$. Since $Q=\mathbb M^{\fin}_{\Omega}(A)$, it follows that
\begin{equation}\label{eq:sxomega}
s_{x\omega}=s_{\omega x}=0 \quad\forall s\in S, ~\omega \in \Omega, ~x\in X\setminus\Omega.
\end{equation}

Set $\Lambda=\Omega\times\N$ and define
\[
 \pi:L\longrightarrow\mathbb M^{\rcf}_{\Lambda}(B),\qquad
 \pi(\ell)_{(\omega,m),(\eta,n)}
   =\bigl(\rho(\ell)_{\omega\eta}\bigr)_{mn}.
\]
Clearly $\pi$ is $K$-linear. Moreover, in view of \eqref{eq:sxomega} we have
\[
 \begin{aligned}
 \pi(\ell\ell')_{(\omega,m),(\eta,n)}
 =\sum_{\zeta\in\Omega}\sum_{k\in\N}
    \bigl(\rho(\ell)_{\omega\zeta}\bigr)_{mk}
    \bigl(\rho(\ell')_{\zeta\eta}\bigr)_{kn}=\bigl(\pi(\ell)\pi(\ell')\bigr)_{(\omega,m),(\eta,n)}
 \end{aligned}
\]
for every $\ell,\ell'\in L$ and $(\omega,m),(\eta,n)\in\Lambda$. Thus $\pi$ is an algebra homomorphism. Since $Q$ contains $E_{\omega\eta}(E_{mn}(c))$ for every
$\omega,\eta\in\Omega$, $m,n\in\N$ and $c\in B$, the image of $\pi$ contains
every elementary matrix $E_{(\omega,m),(\eta,n)}(c)$. Consequently,
\begin{equation*}\label{eq:finite-matrix-ideal}
\mathbb M^{\fin}_{\Lambda}(B)\subseteq\pi(L).
\end{equation*}
Lemma~\ref{prop:matrix-transfer} now yields the required lattice isomorphism and the
identity $\Phi(IJ)=\Phi(I)\Phi(J)+Y$. 

\medskip
\noindent\emph{The remaining implications.}
Consider the finite-dimensional unital algebra
\[
 B_0=K1_3+KE_{12}+KE_{23}+KE_{13}\subseteq\mathbb M_3(K).
\]
The subspaces
\[
 I=KE_{12}+KE_{13},\qquad J=KE_{23}+KE_{13},\qquad
 W=K(E_{12}+E_{23})+KE_{13}
\]
are two-sided ideals of $B_0$. One checks easily that
\[
 \begin{gathered}
 I\cap(J+W)=I\ne KE_{13}=(I\cap J)+(I\cap W),\\
 IJ=KE_{13},\qquad JI=0.
 \end{gathered}
\]
Hence $\Id(B_0)$ is nondistributive and ideal multiplication in $B_0$ is not
commutative.

Assume (T4) and apply it to $B_0$. The resulting lattice isomorphism implies
that $\Id(L)$ is nondistributive, since $\Id(B_0)$ is nondistributive. Thus (T4)$\Rightarrow$(T2).
To prove (T4)$\Rightarrow$(T3), assume that ideal
multiplication in $L$ were commutative. Formula
\eqref{eq:relativeproduct} would then yield
\[
 \Phi(IJ)=\Phi(I)\Phi(J)+Y
          =\Phi(J)\Phi(I)+Y=\Phi(JI).
\]
The injectivity of $\Phi$ would force $IJ=JI$, contrary to the
calculation above.

Finally, if Condition (LPA) holds, then $L$ is isomorphic to an ordinary
Leavitt path algebra by \cite[Theorem~1]{Preusser}. It follows from \cite[Theorems~3.4 and~4.3]{Rangaswamy} that $\Id(L)$
is distributive and ideal multiplication is commutative.
Hence (T2)$\Rightarrow$(T1) and
(T3)$\Rightarrow$(T1). Together with the implications already proved,
this establishes that (T1)--(T4) are equivalent.
\end{proof}

\section{Cardinalities of ideal sets}
A prime ideal of a ring $R$ is proper and prime with respect to products of two-sided
ideals. A left (respectively right) primitive ideal is an annihilator of a simple
unital left (respectively right) $R$-module. We write $\Spec(R)$,
$\Prim_{\mathrm{left}}(R)$ and $\Prim_{\mathrm{right}}(R)$ for the
corresponding sets of ideals.

\begin{corollary}\label{cor:cardinality}
Let $(E,w)$ be a row-finite weighted graph such that $E^0$ and $E^1$ are countable and Condition (LPA) fails. Set $L=L_K(E,w)$. Then
\[
 |\Id(L)|=|\Spec(L)|=|\Prim_{\mathrm{left}}(L)|
 =|\Prim_{\mathrm{right}}(L)|=|K|^{\aleph_0}.
\]
There are also exactly $|K|^{\aleph_0}$ ideals which are both left and right primitive.
\end{corollary}
\begin{proof}
For $B=K[x_0,x_1,\ldots]$ let 
\[\pi:L\longrightarrow\mathbb M^{\rcf}_{\Lambda}(B)\]
be the homomorphism constructed in the proof of Theorem~\ref{thm:main}. Set
\[
 S=\pi(L),\qquad H=\mathbb M^{\fin}_{\Lambda}(B)\lhd S\subseteq \mathbb M^{\rcf}_{\Lambda}(B).
\]
For each tuple $t=(t_n)_{n\in\N}\in K^{\N}$, let
$\operatorname{ev}_t:B\to K$ be the evaluation homomorphism determined
by $\operatorname{ev}_t(x_n)=t_n$. Set
\[
 \mathfrak m_t=\ker\operatorname{ev}_t,\qquad
 H_t=\mathbb M^{\fin}_{\Lambda}(\mathfrak m_t).
\]
Each $H_t$ is an ideal of $S$, since $\mathfrak m_t$ is an ideal of $B$.
Entrywise evaluation defines an algebra homomorphism
\[
 \varepsilon_t:S\longrightarrow\mathbb M^{\rcf}_{\Lambda}(K),\qquad
 \varepsilon_t((b_{\lambda\mu}))
   =(\operatorname{ev}_t(b_{\lambda\mu})).
\]
Let $R_t=\varepsilon_t(S)$. Clearly
\[\mathbb M^{\fin}_{\Lambda}(K)=\varepsilon_t(H)\subseteq \varepsilon_t(S)=R_t\subseteq \mathbb M^{\rcf}_{\Lambda}(K).\]

\medskip
\noindent\emph{The ideals $ \mathcal P_t\lhd L$.}
Let $V=K^{(\Lambda)}$ be the vector space of finite-support columns,
and let $V'$ be the vector space of finite-support rows with the same
index set. Ordinary matrix multiplication makes them left and right
$R_t$-modules. Both modules are unital:
the finite diagonal projection on the support of a vector fixes it on
the appropriate side.

The modules $V$ and $V'$ are simple and faithful. To prove that $V$ is
simple as a left module, let $v\in V$ be nonzero and choose
$\mu\in\Lambda$ with $v_\mu\ne0$.
For every $\lambda\in\Lambda$, multiplication by
$E_{\lambda\mu}(v_\mu^{-1})$ sends $v$ to the standard column with $1$ at 
position $\lambda$ and zeros elsewhere. Thus every nonzero submodule contains all standard
columns. The proof for rows uses
$v'E_{\mu\lambda}((v'_\mu)^{-1})$ in the same way.
A matrix in $R_t$ annihilating every standard column has all columns
zero, and one annihilating every standard row has all rows zero.
This proves faithfulness on both sides.

Define
\begin{equation}\label{eq:primitive-annihilator}
 P_t=\ker\varepsilon_t.
\end{equation}
It follows immediately that
\begin{equation}\label{eq:primitive-intersection}
 P_t\cap H=H_t.
\end{equation}
Since $H/H_t$ is nonzero, this also shows that $P_t$ is proper.

The two modules $V$ and $V'$, regarded as $L$-modules via
$\varepsilon_t\circ\pi:L\to R_t$, are simple and unital, with common
annihilator
\[
 \mathcal P_t:=\pi^{-1}(P_t).
\]
Thus $\mathcal P_t$ is both left and right primitive. It is also prime.
Indeed, an algebra $R$ with a faithful simple left module $M$ is
prime (i.e. the zero ideal of $R$ is prime): for any nonzero ideal $J\lhd R$, the submodule $JM$ is nonzero
and hence equals $M$. If $IJ=0$ for another ideal $I$, then
\[
 IM=I(JM)=(IJ)M=0,
\]
so faithfulness implies $I=0$. Apply this to $S/P_t$, which acts
faithfully and simply on $V$, and pull back along $L\to S/P_t$.

\medskip
\noindent\emph{$|K|^{\aleph_0}$ is a lower bound for the cardinalities.}
Let $s,t\in K^{\N}$ be distinct, and choose $n$ with $s_n\ne t_n$.
Fix $\lambda\in\Lambda$. The matrix
\[
 h=E_{\lambda\lambda}(x_n-t_n)\in H
\]
belongs to $H_t$ but not to $H_s$, since its coefficient evaluates to
$0$ at $t$ and to $s_n-t_n\ne0$ at $s$.
By \eqref{eq:primitive-intersection}, $h\in P_t$ and $h\notin P_s$.
Choose $\ell\in L$ with $\pi(\ell)=h$. Then
\[
 \ell\in\mathcal P_t,\qquad \ell\notin\mathcal P_s.
\]
Thus $t\mapsto\mathcal P_t$ is injective. It follows that 
$|K^{\N}|=|K|^{\aleph_0}$ is a lower bound for the cardinality of each of the ideal sets under consideration.

\medskip
\noindent\emph{$|K|^{\aleph_0}$ is an upper bound for the cardinalities.}
Since $E$ is countable and every edge has finite weight, the defining
generators of $L$ form a countable set. There are countably many finite words in these generators, and these words span
$L$ as a $K$-vector space. Consequently
\[
 \dim_K L\leq\aleph_0,\qquad
 |L|\leq\kappa:=\max(|K|,\aleph_0).
\]
Every ideal is a vector subspace of $L$, so it has a countable
basis and can be spanned by a sequence of elements of $L$, padding
with zeros if necessary. The map sending a sequence to its linear
span therefore has every ideal among its values. Hence
\[
 |\Id(L)|\leq |L|^{\aleph_0}
 \leq\kappa^{\aleph_0}=|K|^{\aleph_0}.
\]
If $K$ is infinite, the last equality follows from $\kappa=|K|$.
If $K$ is finite, then $2\leq|K|<\aleph_0$ and both powers equal
$2^{\aleph_0}$. Since all the other sets of ideals under consideration are
subsets of $\Id(L)$, the assertion of the corollary follows.
\end{proof}

\begin{remark}
The construction of the distinct ideals $\mathcal P_t$ does not use
countability of $E$. Thus every Leavitt path algebra of a row-finite weighted graph failing
Condition (LPA) has at least $|K|^{\aleph_0}$ ideals of each kind in Corollary~\ref{cor:cardinality}.
\end{remark}

\section*[Declaration of generative AI use]{Declaration of generative AI and AI-assisted technologies in the manuscript preparation process}
During the preparation of this work the author used ChatGPT to discuss mathematical ideas and check proofs. After using this tool, the author reviewed and edited the content as needed and takes full responsibility for the content of the published article.


\begin{thebibliography}{9}


\bibitem{Ara}
P. Ara, \emph{Leavitt path algebras of weighted and separated graphs}, J. Aust. Math. Soc. \textbf{115} (2023), 1--25.

\bibitem{AAM} G. Abrams, P. Ara, M. Siles Molina, \textit{Leavitt path algebras}, Lecture Notes in Mathematics {\bf 2191}, Springer, 2017.

\bibitem{DamalerioHazratNam}
R. J. Damalerio, R. Hazrat and T. G. Nam, \emph{The graded Grothendieck group $K_0^{\mathrm{gr}}$ is full for weighted Leavitt path algebras}, preprint (2026), \href{https://arxiv.org/abs/2606.15075v1}{arXiv:2606.15075v1}.

\bibitem{HazratNam}
R. Hazrat and T. G. Nam, \emph{On structural connections between sandpile monoids and weighted Leavitt path algebras}, J. Algebra \textbf{678} (2025), 543--569.

\bibitem{Preusser}
R. Preusser, \emph{Weighted Leavitt path algebras that are isomorphic to unweighted Leavitt path algebras}, Algebr. Represent. Theory \textbf{24} (2021), 403--423.

\bibitem{Survey}
R. Preusser, \emph{Weighted Leavitt path algebras---an overview}, Zap. Nauchn. Sem. POMI \textbf{531} (2024), 157--237.

\bibitem{Rangaswamy}
K. M. Rangaswamy, \emph{The multiplicative ideal theory of Leavitt path algebras}, J. Algebra \textbf{487} (2017), 173--199.

\end{thebibliography}
\end{document}